\documentclass[12pt]{amsart}
\usepackage{amsmath,xcolor}
\usepackage{amssymb}
\input xy
\xyoption{all}

\usepackage{graphicx}

\usepackage[cmtip,all]{xy}

\usepackage{}

\newtheorem{theorem}{Theorem}[section]

\newtheorem{Lemma}[theorem]{Lemma}
\newtheorem{corollary}[theorem]{Corollary}
\newtheorem{Prop}[theorem]{Proposition}
\newtheorem{Thm}[theorem]{Theorem}
\theoremstyle{definition}
\newtheorem{remark}[theorem]{Remark}
\newtheorem{example}[theorem]{Example}
\numberwithin{equation}{section}

\newcommand{\id}{\operatorname{id}}

\newcommand{\Ker}{\operatorname{Ker}}

\newcommand{\Max}{\operatorname{Max}}
\newcommand{\Spec}{\operatorname{Spec}}
\newcommand{\Reg}{\operatorname{Reg}}
\newcommand{\Jac}{\operatorname{Jac}}

\newcommand{\cmat}{\left(\begin{array}}
\newcommand{\fmat}{\end{array}\right)}

\newcommand{\gm}{\mathfrak{m}}
\newcommand{\gp}{\mathfrak{p}}

\newcommand{\ga}{\mathfrak{a}}
\newcommand{\gb}{\mathfrak{b}}
\newcommand{\gc}{\mathfrak{c}}
\newcommand{\gi}{\mathfrak{i}}
\newcommand{\gj}{\mathfrak{j}}
\newcommand{\gh}{\mathfrak{h}}
\newcommand{\f}{\mathfrak}
\newcommand{\RR}{A\bowtie^{f,g}(\gb,\gc)}

\begin{document}

 \title[]{On bi-amalgamated constructions}
  

 \author{Federico Campanini}
\address{Dipartimento di Matematica ``Tullio Levi Civita", Universit\`a di Padova, 35121 Padova, Italy}
\email{federico.campanini@math.unipd.it, carmelo@math.unipd.it}
 
\author{Carmelo Antonio Finocchiaro}



\begin{abstract} 
	Let $f:A\longrightarrow B, g:A\longrightarrow C$ be ring homomorphisms and let $\f b$ (resp., $\f c$) be an ideal of $B$ (resp., $C$) satisfying $f^{-1}(\f b)=g^{-1}(\f c)$. Recently Kabbaj, Louartiti and Tamekkante defined and studied the following subring 
	$$
	\RR :=\{(f(a)+b, g(a)+c)\mid a\in A, b\in\f b, c\in \f c \}
	$$
	of $B\times C$, called the bi-amalgamation of $A$ with $(B,C)$ along $(\f b, \f c)$, with respect to $(f,g)$. This ring construction is a natural generalization of the amalgamated algebras, introduced and studied by D'Anna, Finocchiaro and Fontana. 
	
	The aim of this paper is to continue the investigation started by Kabbaj, Louartiti and Tamekkante, by providing a deeper insigt on the ideal-theoretic structure of bi-amalgamations. 
\end{abstract}

\maketitle

\section{Introduction}
Let $A$ be a ring and let $\f a$ be an ideal of $A$. In \cite{DF} the authors introduced the following subring 
$$
A\bowtie \f a:=\{(a,a+x)\mid a\in A,x\in \f a \}
$$
of the direct product $A\times A$, and named it \emph{the amalgamated duplication of $A$ along $\f a$}. This ring construction arises as a sort of variant of the Nagata idealization $A(+)X$  of an $A$-module $X$: it is defined by endowing the product $A\times X$ with the ring structure where addition is induced componentwise from additions of $A$ and $X$, and multiplication is defined  by setting
$$
(a,x)(b,y):=(ab, ay+bx), \quad \mbox{for all}\quad (a,x),(b,y)\in A(+)X;
$$
in such a way the ring $A$ is identified as a subring of $A(+)X$ and the $A$-module $X$ is isomorphic to the ideal $0\times X$ of $A(+)X$. It can be proved that if $\f a$ is a nilpotent ideal of index 2, then the duplication $A\bowtie \f a$ is canonically isomorphic to the idealization $A(+)\f a$, but in general this is not the case, since if $A$ is a reduced ring, then $A\bowtie \f a$ is a reduced ring, while $(0)\times X$ is a nilpotent ideal of index 2 of $A(+)X$, whenever $X\neq 0$. Among motivations for studying duplications of rings along ideals, it is worth noting that in \cite{D} the author showed that the amalgamated duplication of an algebroid curve along a regular multiplicative ideal always yields  an algebroid curve. 

As it is immediately seen, whenever $\f a$ is a nonzero ideal of $A$ the ring $A\bowtie \f a$ is not an integral domain. A new ring construction, which is more general than the amalgamated duplication and can be an integral domain, was introduced in \cite{DFF} as follows: starting from a ring homomorphism $f:A\longrightarrow B$ and from an ideal $\f b$ of $B$, the author defined \emph{the amalgamation of $A$ and $B$ along $\f b$, with respect to $f$} to be the following subring
$$
A\bowtie^f\f b:=\{(a,f(a)+b)\mid a\in A, b\in\f b \}
$$
of the direct product $A\times B$. Obviously, if $f={\rm Id}_A:A\longrightarrow A$ and $\f a$ is an ideal of $A$, then $A\bowtie^f\f a$ is the classical amalgamated duplication. This new ring construction, whose genesis is in some sense related to the procedure due to Dorroh of endowing a ring (without unity) with a unity (see, for instance, \cite[Remark 2.2]{DF}), is a powerful tool for unifying several classical classes of rings, as $D+\f m$ domains, rings of the type $A+XB[X]$, $A+XB[\![X]\!]$, CPI extensions. The fact that rings of the form $A\bowtie^f\f b$ arise with a peculiar fiber product structure allowed to provide a complete description of the order theoretic and  topological properties of its prime spectrum and furthermore to find bounds for its Krull dimension \cite{DFF1}. In \cite{FCA1} and, more recently, in \cite{ASS} the authors investigated the transfer of Pr\"ufer, Gaussian and arithmetical conditions to amalgamated constructions. 

A new class of rings, which covers that of amalgamated algebras as a particular case, was recently introduced in \cite{KLT}: given ring homomorphisms $f:A\longrightarrow B, g:A\longrightarrow C$ and given ideals $\f b$ of $B$ and $\f c$ of $C$ such that $f^{-1}(\f b)=g^{-1}(\f c)$, the subring 
$$
\RR:=\{(f(a)+b,g(a)+c)\mid a\in A, b\in\f b, c\in\f c \}
$$
of $B\times C$ is called \textit{the bi-amalgamation of $A$ with $(B,C)$ along $(\f b,\f c)$ with respect to $(f,g)$}. In \cite{KLT} the authors investigated about the basic ring-theoretic properties of bi-amalgamations and provided, among the other things, a description of its prime ideal structure. In \cite{KMM} the transfer of the arithmetical property to bi-amalgamated algebras is discussed and some application is provided. 

The aim of this paper is to continue the investigation on the ideal-theoretic structure of bi-amalgamated algebras, started in \cite{KLT}, and its outline is as follows: in Section 2, after presenting some basic properties of the ring $\RR$, we apply its fiber product presentation for giving further remarks on the topological nature of its prime spectrum and a precise characterization of when it is a local ring. One of the main results of \cite{KLT} is that the ring $\RR$ is Noetherian if and only if the rings $f(A)+\f b$ and $g(A)+\f c$ are Noetherian. This characterization could be not so helpful on providing applications, since Noetherianity is not tested directly via the given rings. Thus is Section 3 we provide some more direct criterion for Noetherianity of a relevant class of bi-amalgamations. In Section 4 we investigate the problem of when the ring $\RR$ is Pr\"ufer. A precise answer is given in case both ideals $\f b,\f c$ are regular. The general case is much more difficult and  necessary conditions and sufficient conditions for $\RR$ to be Pr\"ufer are given under a good behaviour of the regular elements of the factor ring $A/f^{-1}(\f b)$. Finally, in Section 5 we provide conditions for a bi-amalgamation to be a Gauss ring and new examples. 
\section{Notation and preliminaries}
\noindent
In what follows, any ring is assumed to be commutative with multiplicative identity 1. Any ring homomorphism $f:A\longrightarrow B$ sends the identity of $A$ to that of $B$, and $f^\star:\Spec(B)\longrightarrow \Spec(A)$ denotes the canonical continuous map, with respect to the Zariski topology, induced by $f$. An element of a ring is regular if it is not a zero-divisor. An ideal is regular if it contains a regular element. The set of all regular elements of a ring $A$ will be denoted by $\Reg(A)$.

Recall that, given ring homomorphisms $\alpha:A\longrightarrow C,\beta:B\longrightarrow C$, the \textit{fiber product of $\alpha$ and $\beta$} is the following subring of $A\times B$:
$$
\alpha \times_C\beta :=\{(a,b)\in A\times B\mid \alpha(a)=\beta(b) \}.
$$
That of fiber products is the tipical algebraic structure intevening in the classical operation of attaching spectral spaces. Thus this type of constructions have a prominent role to produce rings whose prime spectra have to satisfy prescribed conditions. For a deeper understanding of this circle of ideas see, for instance, \cite{FO1}. 

Let $f:A\longrightarrow B$ and $g: A \longrightarrow C$ be two ring homomorphisms and let $\gb$ and $\gc$ be two ideals of $B$ and $C$, respectively, such that $\f i_0:=f^{-1}(\gb)=g^{-1}(\gc)$. According to \cite{KLT}, \textit{the bi-amalgamation of $A$ with $(B,C)$ along $(\gb,\gc)$ with respect to $(f,g)$} is the subring
$$
\RR :=\{(f(a)+b,\ g(a)+c)\mid a \in A,b\in\f b, c\in\f c \}
$$
of $B\times C$. First of all, it is worth noting that the usual amalgamation 
$$
A\bowtie^f \f b:=\{(a, f(a)+b)\mid a\in A,b\in\f b \}
$$
is a very particular bi-amalgamation since, in view of \cite[Example 2.1]{KLT}, we  have $$A\bowtie^f \f b=A\bowtie^{f,{\rm Id}_A}(\f b, \f i_0),$$
where ${\rm Id}_A:A\longrightarrow A$ is the identity. 

Following the same pattern of \cite[Proposition 3.2]{KLT}, it is easily observed that the bi-amalgamation $\RR$ is isomorphic to the fiber product of the ring homomorphisms $ i_{f,g}$ and $ \pi$ (that make the following diagram commute)
\begin{equation}
\xymatrix{
\RR \ar[r] \ar[d]
 &  A/\gi_0 \ar[d]^{i_{f,g}} \\
B\times C \ar[r]^\pi & \frac{B}{\gb} \times \frac{C}{\gc}, 
}
\end{equation}
where $\pi: B\times C\longrightarrow B/\gb\times C/\gc$ is the canonical projection and $i:=i_{f,g}:A/\gi_0\longrightarrow B/\gb\times C/\gc$ is the canonical ring embedding defined by setting $i_{f,g}(a+\f i_0):=(f(a)+\f b, g(a)+\f c)$, for any $a+\f i_0\in A/\f i_0$.

In other words, we have $$\RR=\pi^{-1}(i_{f,g}(A/\f i_0)).$$
We summarize in the next remark some basic and useful properties of any bi-amalgamation. 
\begin{remark} \label{prime} Let $f:A\longrightarrow B$ and $g: A \longrightarrow C$ be two ring homomorphisms and let $\gb$ and $\gc$ be two ideals of $B$ and $C$, respectively, such that $f^{-1}(\gb)=g^{-1}(\gc)$. Set also $\gi_0:=f^{-1}(\gb)=g^{-1}(\gc)$.
\begin{enumerate}
	\item If $\ga$ is an ideal of $A$, we can define
	$$
	\ga\bowtie^{f,g}(\gb,\gc):=\{(f(p)+b,g(p)+c)\mid p \in \ga,\ (b,c) \in \gb\times \gc\}
	$$
	which is an ideal of $\RR$ containing $\gb\times \gc$ (see, for instance,  \cite[Section 4]{KLT}). Moreover, if $\gj$ and $\gj'$ are ideals of $B$ and $C$ respectively, then also the sets
	$$
\f j^{\sharp_B}:= \{(f(a)+b,\ g(a)+c)\mid a \in A,\ (b,c) \in \gb\times \gc ,f(a)+b\in \gj \}
	$$
	and
	$$
	\f j'^{\sharp_C}:=\{(f(a)+b,\ g(a)+c)\mid a \in A,\ (b,c) \in \gb\times \gc ,g(a)+c\in \gj' \}
	$$
	are ideals of $\RR$, since they are the contractions of the ideals $\f j\times C$ and $B\times \f j'$, respectively, of $B\times C$ in $\RR$. 
		\item Set $\f k:=\f k_{f,g}:=\Ker(f)\cap \Ker(g)$ and note that $f$ and $g$ induce a natural ring embedding $\iota:=\iota_{f,g}:A/\f k\longrightarrow \RR$. It is defined by setting $\iota(a+\f k):=(f(a),g(a))$, for any $a+\f k\in A/\f k$. 
		\item Note that if $\f b$ and $\f c$ are finitely generated $A$-modules (where the $A$-module structures are defined via $f,g$, respectively) then the ring embedding $\iota:A/\f k\longrightarrow\RR$ is finite. As a matter of fact, if $\{b_1,\ldots, b_n \}$ (resp., $\{c_1,\ldots, c_m\}$) is a set of generators of $\f b$ (resp., $\f c$) as an $A$-module, 
		$$
		\{(1,1),(b_i,0),(0,c_j):1\leq i\leq n, 1\leq j\leq m \}
		$$
		is a set of generators of $\RR$ as a $A/\f k$-module. 
		\item If $f$ and $g$ are finite ring homomorphisms, then the ring extension $\RR \subseteq B\times C$ is finite. Indeed, if $\{x_1,\ldots,x_n\}$ (resp., $\{y_1,\ldots,y_m \}$) is a set of generators of $B$ (resp., $C$) as an $A$-module, then 
		$$
		\{(x_i,0), (0,y_j):1\leq i\leq n, 1\leq j\leq m \}
		$$
		is a set of generators of $B\times C$ as a $\RR$-module.
\end{enumerate}

\end{remark}
From now on we preserve the notation of Remark \ref{prime}.	
The following property regarding the order structure of the ideals of $\RR$ will be useful. 
\begin{Lemma}
If $\ga_1$ and $\ga_2$ are two ideals of $A$ containing $\gi_0$ such that $\ga_1\bowtie^{f,g}(\gb,\gc)\subseteq \ga_2\bowtie^{f,g}(\gb,\gc)$, then $\ga_1\subseteq \ga_2$.
\end{Lemma}

\begin{proof}
	Let $x$ be an element of $\ga_1$. Since $(f(x),g(x))\in \ga_1\bowtie^{f,g}(\gb,\gc)\subseteq \ga_2\bowtie^{f,g}(\gb,\gc)$, there exist elements $y \in \ga_2$, $b \in \gb$ and $c\in \gc$ such that $(f(x),g(x))=(f(y)+b,g(y)+c)$. In particular, $x-y\in \gi_0\subseteq \ga_2$, which implies that $x \in \ga_2$.
\end{proof}

In \cite{KLT} a presentation of the prime spectrum of a biamalgamation is given, in terms of prime ideals of the rings $f(A)+\f b$ and $g(A)+\f c$. In the following proposition we provide an equivalent description of the prime spectrum that is more useful in the applications. 
\begin{Prop}\label{prime-spectrum}
 The following statements hold. 
\begin{enumerate}
	\item The canonical surjective ring homomorphism $$p:\RR \longrightarrow A/\f i_0, \qquad (f(a)+b,g(a)+c)\mapsto a+\f i_0$$ induces the closed topological embedding $$p^\star:V(\f i_0)\longrightarrow \Spec(\RR), \qquad \f p\mapsto \f p\bowtie^{f,g}(\f b,\f c)$$ 
	establishing a homeomorphim of $V(\f i_0)$ and the image  $p^\star(V(\f i_0))=V(\f b\times \f c)\subseteq \Spec(\RR)$.
	\item The inclusion $i:\RR \longrightarrow B\times C$ induces a continuous map $i^\star: \Spec(B\times C)\longrightarrow \Spec(\RR)$, defined by contraction. The mapping $i^\star$ induces by restriction a homeomorphism 
	$$
	\Spec(B\times C)\setminus V(\f b\times \f c)\longrightarrow \Spec(\RR)\setminus V(\f b\times \f c)
	$$ 
	defined by
	$$
	\f q\times C\mapsto \f q^{\sharp_B}, \qquad \qquad B\times \f q'\mapsto \f q'^{\sharp_C}
	$$
	for any $\f q\in\Spec(B)\setminus V(\f b)$ and $\f q'\in \Spec(C)\setminus V(\f c)$.
	\item The homeomorphisms defined in (1) and (2) preserves maximality. 
	\item \label{local} $\RR$ is a local ring if and only if $A/\gi_0$ is a local ring, $\gb \subseteq \Jac(B)$ and $\gc \subseteq \Jac(C)$. 
\end{enumerate}
\end{Prop}
\begin{proof}
We keep in mind that $\RR$ is canonically isomorphic to the fiber product of the natural maps $i_{f,g}:A/\f i_0\longrightarrow B/\f b\times C/\f c$ and $\pi:B\times C \longrightarrow  B/\f b\times C/\f c$. Now the conclusion is an immediate consequence of \cite[Theorem 1.4]{FO1} and \cite[Corollary 2.4]{DFF2}. 
\end{proof} 
\section{Noetherian bi-amalgamations}
In \cite{KLT} the authors investigated about when a bi-amalgamated algebra is a Noetherian ring. By using its pullback structure, they proved that the ring $\RR$ is Noetherian if and only if $f(A)+\f b, g(A)+\f c$ are Noetherian rings \cite[Proposition 4.2]{KLT}. It is worth noting that this characterization is often difficult to be applied since Noetherianity of $\RR$ is not directly related to the given data ($A,B,\f b, \f c,f,g$) but to the rings $f(A)+\f b, g(A)+\f c$. Thus the next goal is to describe more explicitly Noetherianity of $\RR$ in some relevant cases. 
\begin{Prop}\label{Noeth}
 Assume that $f$ and $g$ are finite homomorphisms. Then the following conditions are equivalent:
	\begin{enumerate}
		\item
		$\RR$ is a Noetherian  ring;
		\item
		 $B$ and $C$ are Noetherian rings;
		 \item $A/\f k$ is a Noetherian ring.
	\end{enumerate}
\end{Prop}

\begin{proof}
	First notice that since $f$ and $g$ are finite morphisms, then the ring extension $\RR \subseteq B\times C$ is finite (Remark \ref{prime}(4)). Now the equivalence
(1)$\Longleftrightarrow$(2) follows from Eakin-Nagata's Theorem (and from the fact that a finite extension of a Noetherian ring is a Noetherian ring). 

(2)$\Longrightarrow$(3). Since $f,g$ are finite, condition (2) and Eakin-Nagata's Theorem imply that $f(A),g(A)$ are Noetherian ring. It immediately follows that $A/\f k$ is a Noetherian ring (see, for instance, \cite[Exercise 3.1]{MAT}).

(3)$\Longrightarrow$(2). Since $f,g$ are finite, then the canonical ring homomorphisms $f_1:A/\f k\longrightarrow B, g_1:A/\f k\longrightarrow C$ are finite too. Since $A/\f k$ is Noetherian, it follows that $B,C$ are Noetherian rings too. 
\end{proof}
\begin{Prop}\label{b-c-A-finite}
Assume that $\f b$ and $\f c$ are finitely generated as $A$-modules. Then $\RR$ is Noetherian if and only if $A/\f k$ is Noetherian. 
\end{Prop}
\begin{proof}
Apply Remark \ref{prime}(3) and Eakin-Nagata's Theorem.
\end{proof}
As a special case of the previous propositions, we now recover the following result about amalgamated algebras.

\begin{corollary} \cite[Proposition 5.7 (d)]{DFF}
	Let $f: A\longrightarrow B$ be a ring homomorphism, $\gb$ an ideal of $B$ and let $A\bowtie^f \gb$ denote the amalgamation of $A$  with $B$ along $\gb$ with respect to $f$.  Assume that one of the following conditions is satisfied.
	\begin{enumerate}
		\item $\f b$ is a finitely generated $A$-module.
		\item $f$ is a finite ring homomorphism.
	\end{enumerate}
	Then $A\bowtie^f \gb$ is a Noetherian ring if and only if $A$ is a Noetherian ring.
\end{corollary}

\begin{proof}
	Notice that the ring $A\bowtie^f \gb$ is isomorphic to the bi-amalgamation $A\bowtie^{\id_A,f}(\gi_0, \gb)$ (where $\f i_0:=f^{-1}(\f b)$). Then, apply Propositions \ref{Noeth} and \ref{b-c-A-finite} to get the conclusion, keeping in mind that in the present case $\f k=0$.
\end{proof}

\begin{example}
	\begin{enumerate}
		\item
		The fact that $\RR$ is Noetherian does not imply that $A$ is a Noetherian ring, even if $f$ and $g$ are finite morphisms. Indeed, we can consider a non-Noetherian ring $A$ with an ideal $\gi_0$ such that the quotient ring $A/\gi_0$ is Noetherian. Set $B=C=A/\gi_0$, $\gb=\gc=0$ and let $f, g$ be the canonical projections. Then, by \cite[Proposition 4.1(3)]{KLT}, the ring $\RR$ is isomorphic to $A/\gi_0$.
		
		\item
		It is possible to find examples of bi-amalgamations in which $A/\gi_0$ is a Noetherian ring, $f$ is a finite morphism but $f(A)+\gb$ is not Noetherian (and so, in particular, $\RR$ is not a Noetherian ring \cite[Proposition 4.2]{KLT}). For instance, let $A$ be a non-Noetherian ring with an ideal $\gi_0$ such that the quotient ring $A/\gi_0$ is Noetherian and let $f=\id_A$. Then $f(A)+\gb=A$ is not a Noetherian ring. 
		\item  Under the assumption of Proposition \ref{b-c-A-finite} (namely, $\f b$ and $\f c$ are finitely generated as $A$-modules) it is not true that Noetherianity of $\RR$ can ben tested with Noetherianity of $B$ and $C$ (as it happens in Proposition \ref{Noeth}). For instance, let $A$ be a Noetherian ring and let $B:=C:=A^{\mathbb N}$ be the product of countably many copies of $A$. Let $f:=g:A\longrightarrow B$ be the diagonal embedding and let $\f b:=\f c$ be the principal ideal of $B$ generated by $(1,0,\ldots)$. Clearly, $\f b$ is principal as an $A$-module too, and thus $\RR$ is Noetherian, in view of Proposition \ref{b-c-A-finite}, but obviously $B$ is not Noetherian. 
		\item When one of the ring homomorphisms $f,g$ fails to be finite or when one of the ideals $\f b,\f c$ fails to be finitely generated as an $A$-module, then Noetherianity of the ring $A/\f k$ is not a necessary condition for $\RR$ to be Noetherian. For instance, consider the polynomial ring $B:=\mathbb Z[T]$, its subring $A:=\mathbb Z+2T\mathbb Z[T]$ and let $f:=g:A\longrightarrow B$ be the inclusion. If $\f b:=\f c:=T\mathbb Z[T]$, it immediately follows that $f(A)+\f b=\mathbb Z[T]$ and thus $\RR$ is a Noetherian ring, by \cite[Proposition 4.2]{KLT}, but $A(=A/\f k)$ is not Noetherian, in view of \cite[Example 5.1]{DFF}.
	\end{enumerate}
\end{example}

\section{Pr\"ufer bi-amalgamations}
Let $A$ be a ring and let $Q$ be the total ring of fractions of $A$. An ideal $\f a$ of $A$ is called to be \emph{invertible} if there exists an $A$-submodule $F$ of $Q$ such that $\f aF=A$. It is well known that, if $\f a$ is invertible, then it is finitely generated and regular and $F$ is uniquely determined, namely 

$$F=(A:\f a):=\{x\in Q\mid x\f a\subseteq A \}.$$
Following \cite{Gr}, $A$ is called \emph{a Pr\"ufer ring} if any finitely generated regular ideal of $A$ is invertible. For an integral domain $D$ the Pr\"ufer property is equivalent to require that $D$ is locally a valuation domain. The analog fact for rings with zero-divisors is summarized in the next classical result. First, recall that if $\f p$ is a prime ideal of a ring $A$, the pair $(A,\f p)$ is said to have \emph{the regular total order property} if, whenever $\f a,\f b$ are ideals of $A$, one at least of which is regular, then the ideals $\f aA_{\f p}, \f bA_{\f p}$ are comparable.
\begin{Thm}\cite[Theorem 13]{Gr} For a ring $A$ the following conditions are equivalent. 
\begin{enumerate}
\item $A$ is a Pr\"ufer ring. 
\item For any maximal ideal $\f m$ of $A$, the pair $(A,\f m)$ has the regular total order property. 
\end{enumerate}
\end{Thm}

Let $A$ be a ring and $T$ be an indeterminate over $A$. Recall that the content of a polynomial $f\in A[T]$ is the ideal $c_A(f)$ of $A$ generated by the coefficients of $f$. We say that the polynomial $f$ is \emph{a Gauss polynomial over $A$} if $c_A(fg)=c_A(f)c_A(g)$, for any polynomial $g\in A[T]$. The ring $A$ is \emph{ a Gauss (or Gaussian) ring} if any polynomial $f\in A[T]$ is a Gauss polynomial. We summarize now well-known results about the interplay between invertibility of ideals and Gauss polynomials.

\begin{theorem}\label{gauss-known}
	Let $A$ be a ring, let $T$ be an indeterminate over $A$ and let $f(T)\in A[T]$. 
	\begin{enumerate}
		\item  If $c_A(f)$ is locally principal, then $f(T)$ is a Gauss polynomial \cite{HT}. 
		\item If $f(T)$ is a Gauss polynomial and $c_A(f)$ is a regular ideal of $A$, then $c_A(f)$ is an invertible ideal of $A$ \cite[Theorem 4.2(2)]{Ba-Gl}. 
	\end{enumerate} 
\end{theorem}
From the last part of the previous theorem it immediately follows that any Gauss ring is Pr\"ufer. As it was observed in \cite{Ba-Gl-JA}, there are Pr\"ufer rings that are not Gaussian. However, such notions are equivalent for rings whose total ring of fractions is absolutely flat. In particular, an integral domain is a Gauss domain if and only if it is Pr\"ufer. 

In what follows we investigate the problem of when a bi-amalgamation is a Pr\"ufer ring. Our first goal is to provide necessary conditions.

 Let $\pi: A \longrightarrow A/\gi_0$ be the canonical projection. Consider the following ``regular properties'' for $f$ and $g$:
\begin{equation}\label{*}\tag{*}
f(\pi^{-1}(\Reg(A/\gi_0)))\subseteq \Reg(B) \mbox{ and }
g(\pi^{-1}(\Reg(A/\gi_0)))\subseteq \Reg(C)
\end{equation}
and
\begin{equation}\label{**}\tag{**}
f(\Reg(A))\subseteq \Reg(B) \mbox{ and } g(\Reg(A))\subseteq \Reg(C).
\end{equation}

\begin{Prop}\label{gauss}
Assume that condition (\ref{*}) holds. If $\RR$ is a Pr\"ufer ring, then so is $A/\gi_0$.
\end{Prop}

\begin{proof}
During the proof, we denote $\RR$ and $A/\gi_0$ simply by $R$ and $R_0$ respectively. Let $\gh=(x_0+\gi_0,x_1+\gi_0,\dots, x_n+\gi_0)$ be a regular finitely generated ideal of $R_0$. Set $\eta_i:=(f(x_i),g(x_i)) \in R$ for every $i=0,\dots, n$. Then condition (\ref{*}) implies that $\gh':=(\eta_0,\eta_1,\dots,\eta_n)$ is a regular finitely generated ideal of $R$. In particular the polynomial $P(T):=\sum_{i=0}^n \eta_i T^i \in R[T]$ is a Gauss polynomial. Our aim is to show that also the polynomial $F(T):=\sum_{i=0}^n (x_i+\gi_0) T^i \in R_0[T]$ is a Gauss polynomial. Let $G(T)=\sum_{j=0}^m (y_i+\gi_0)T^j$ be any other polynomial in $R_0[T]$. It suffices to prove that $c_{R_0}(F)c_{R_0}(G)\subseteq c_{R_0}(FG)$. If $a+\gi_0 \in c_{R_0}(F)c_{R_0}(G)$ and $Q(T):=\sum_{j=0}^m (f(y_j),g(y_j))T^j\in R[T]$, then $(f(a)+b,g(a)+c) \in c_R(P)c_R(Q)=c_R(PQ)$ for suitable elements $b\in \gb$ and $c \in \gc$. Looking only at the first coordinate of $R$, we get that there exist $a_0,\dots,a_{n+m} \in A$ and  $b_0,\dots, b_{n+m} \in \gb$ such that
$$
f(a)+b=\sum_{k=0}^{n+m} \left((f(a_k)+b_k)\sum_{i+j=k}f(x_i y_j)\right),
$$
so, in particular,
$$
f(a)-\sum_{k=0}^{n+m}\left( f(a_k)\sum_{i+j=k}f(x_i y_j) \right) \in \gb.
$$
It means that $a-\sum_{k=0}^{n+m}\left( \sum_{i+j=k} x_i y_j\right)a_k \in \gi_0$, which implies that $a+\gi_0 \in c_{R_0}(FG)$.
\end{proof}

\begin{Prop}\label{sabato}
Assume that $\RR$ is a Pr\"ufer ring, let $\ga$ be an ideal of $A$ and let $p:A\longrightarrow A/\f a$ be the canonical projection. If
$f(p^{-1}(\Reg(A/\ga)))\subseteq \Reg(B)$ and $g(p^{-1}(\Reg(A/\ga)))\subseteq \Reg(C)$,
then $A/\ga$ is a Pr\"ufer ring whenever one of the following conditions holds:
\begin{enumerate}
\item
$\gi_0\subseteq \ga$;
\item
$g$ is surjective and $\Ker(g)\subseteq \ga$;
\item
$\f c\subseteq g(A)$ and $\Ker(g)\subseteq \ga$.
\end{enumerate}
\end{Prop}

\begin{proof}
It suffices to notice that the proof of Theorem \ref{gauss} can be easily adapted in all three cases of the statement.
\end{proof}
\begin{corollary}\cite[Proposition 4.2]{FCA1}
Let $f:A\longrightarrow B$ be a ring homomorphism such that $f(\Reg(A))\subseteq \Reg(B)$ and let $\f b$ be an ideal of $B$ such that $A\bowtie^f\f b$ is a Pr\"ufer ring. Then $A$ is a Pr\"ufer ring. 
\end{corollary}
\begin{proof}
	Apply case (2) of Proposition \ref{sabato} to $g:={\rm Id}_A:A\longrightarrow A$, $\f a:=0$ and $\f c:=\f i_0$. Then such a bi-amalgamation is $A\bowtie^f\f b$, as we have seen at the beginning of Section 2.
\end{proof}
\begin{example}\label{example1}
Let $\mathbf{k}$ be a field and set $A=\mathbf{k}[X,Y]$, $\gi_0=(Y)$, $B=C=A/\gi_0$, $\gb = \gc =0$ and $f=g=\pi: A \longrightarrow A/\gi_0$. It is easy to see that condition (\ref{*}) holds, $\RR \cong A/\gi_0$ is a Pr\"ufer ring, but $A$ is not a Pr\"ufer ring.
\end{example}
We now recall the following known fact regarding the local structure of a bi-amalgamated algebra.
\begin{Prop}\cite[Proposition 5.7]{KLT}\label{localization}
Let $\gp$ be a prime ideal of $A$ containing $\gi_0$. Consider the multiplicative subsets $S_\gp:=f(A\setminus \gp)+\gb$ of $B$ and $T_\gp:=g(A\setminus \gp)+\gc$ of $C$. Let $f_\gp:A_\gp\longrightarrow B_{S_\gp}$ and $g_\gp: A_\gp\longrightarrow C_{T_\gp}$ be the ring homomorphisms induced by $f$ and $g$. Then
$$
f^{-1}(\gb B_{S_\gp})=g^{-1}(\gc C_{T_\gp})=\gi_0 A_\gp
$$
and
$$
\RR_{(\gp \bowtie^{f,g}(\gb,\gc))}\cong A_\gp\bowtie^{f_\gp,g_\gp}(\gb B_{S_\gp},\gc C_{T_\gp}).
$$
\end{Prop}

\begin{Prop}
Assume that condition (\ref{**}) holds and preserve the notation of Proposition \ref{localization}. If $\RR$ is a Pr\"ufer ring, then $\gb B_{S_\gm}=f_\gm(\frac{r}{1})\gb B_{S_\gm}$ and $\gc C_{T_\gm}=g_\gm(\frac{r}{1})\gc C_{T_\gm}$ for every $\gm \in \Max(A)\cap V(\gi_0)$ and every regular element $r \in A$.
\end{Prop}

\begin{proof}
Fix a maximal ideal $\gm$ of $A$ containing $\gi_0$ and a regular element $r$ of $A$. Set $\sigma=r/1\in A_\gm$. Condition (\ref{**}) implies that $(f_\gm(\sigma),g_\gm(\sigma))$ is a regular element of the local ring $\hat{R}:=\RR_{(\gm\bowtie^{f,g}(\gb,\gc))}$. Fix an element $\tau \in \gb B_{S_\gm}$. By hypothesis, $(\RR,\gm\bowtie^{f,g}(\gb,\gc))$ has the regular total order property, so that, in particular, the principal ideals of $\hat{R}$ generated respectively by $(f_\gm(\sigma),g_\gm(\sigma))$ and $(\tau,0)$ are comparable. Using the fact that $g_\gm(\sigma)\neq 0$, it is immediate to check that the inclusion $(\tau,0)\hat{R}\subseteq  (f_\gm(\sigma),g_\gm(\sigma))\hat{R}$ holds. Thus, there exist elements $\alpha \in A_\gm, \beta \in \gb B_{S_\gm}$ and $\gamma \in \gc C_{T_\gm}$ such that 
$$
(\tau,0)=(f_\gm(\sigma),g_\gm(\sigma))(f_\gm(\alpha)+\beta,g_\gm(\alpha)+\gamma).
$$
From $g_\gm(\sigma)(g_\gm(\alpha)+\gamma)=0$ we get $g_\gm(\alpha)+\gamma=0$, and so $\alpha \in \gi_0 A_\gm$. In particular, $\tau=f_\gm(\sigma)(f_\gm(\alpha)+\beta)\in f_\gm(\sigma)\gb B_{S_\gm}$. It follows that $\gb B_{S_\gm}=f_\gm(\frac{r}{1})\gb B_{S_\gm}$. The equality $\gc C_{T_\gm}=g_\gm(\frac{r}{1})\gc C_{T_\gm}$ can be proved similarly.
\end{proof}

\begin{remark}\label{remark1}
There is an analogous of the previous proposition in the situation in which condition (\ref{*}) holds. In this case, if $\RR$ is a Pr\"ufer ring, then $\gb B_{S_\gm}=f_\gm(\frac{r}{1})\gb B_{S_\gm}$ and $\gc C_{T_\gm}=g_\gm(\frac{r}{1})\gc C_{T_\gm}$ for every $\gm \in \Max(A)\cap V(\gi_0)$ and every element $r \in A$ such that $r+\gi_0$ is a regular element in $A/\gi_0$. The proof is exactly the same, since also condition (\ref{*}) implies that the element $(f_\gm(\sigma),g_\gm(\sigma))$ is regular in $\RR_{(\gm\bowtie^{f,g}(\gb,\gc))}$.
\end{remark}

\begin{Prop}\label{prufer-regular}
Assume that $\gb$ and $\gc$ are regular ideals. Then the following conditions are equivalent:
\begin{enumerate}
\item
$\RR$ is a Pr\"ufer ring;
\item
$B, C$ are Pr\"ufer rings and $\gb=B$.
\end{enumerate}
\end{Prop}

\begin{proof}
First notice that $\gb=B$ if and only if $\gc=C$. In particular, if $B, C$ are Pr\"ufer rings and $\gb=B$, then $\RR\cong B\times C$ is a Pr\"ufer ring and the implication $(2)\Rightarrow (1)$ holds.

Now assume that $\RR$ is a Pr\"ufer ring and suppose that $\gb\neq B$. Let  $\gp$ be a maximal ideal of $A$ containing $\gi_0$ (notice that $\gi_0\neq A$, since $\gb\neq B$). Using the same notations of Proposition \ref{localization}, it is straightforward to show that $\gb B_{S_\gp} \neq B_{S_\gp}$. Moreover, if $b \in \gb$ and $c \in \gc$ are regular elements, then so are $\beta:=b/1 \in \gb B_{S_\gp}$ and $\gamma:=c/1 \in \gc C_{T_\gp}$. In particular, $(\beta,\gamma)$ is a regular element of $\hat{R}:=\RR_{(\gp \bowtie^{f,g}(\gb,\gc))}$. Since $(\RR,\gp \bowtie^{f,g}(\gb,\gc))$ has the regular total order property, it follows that the principal ideals $(\beta,\gamma)\hat{R}$ and $(\beta,0)\hat{R}$ are comparable. It is easy to see that $(\beta,0)\hat{R}\subseteq (\beta,\gamma)\hat{R}$ must be the case. It follows that there exist elements $\alpha\in A_\gp,\beta' \in \gb B_{S_\gp}$ and $\gamma' \in \gc C_{T_\gp}$ such that
$$
(\beta,0)=(f_\gp(\alpha)+\beta',g_\gp(\alpha)+\gamma')(\beta,\gamma).
$$
Looking at the equality $(g_\gp(\alpha)+\gamma')\gamma=0$ and using the fact that $\gamma$ is a regular element, we get that $\alpha \in \gi_0 A_\gp$, and so $f_\gp(\alpha)+\beta'\in \gb B_{S_\gp}$. Moreover, since also $\beta$ is regular, $f_\gp(\alpha)+\beta'=1$, that is $\gb B_{S_\gp} = B_{S_\gp}$, a contradiction. It follows that $\gb=B$. Since $\RR\cong B\times C$ is a Pr\"ufer ring, then so are $B$ and $C$.
\end{proof}
\begin{corollary}\cite[Theorem 3.1]{FCA1}
Let $f:A\longrightarrow B$ be a ring homomorphism and let $\f b$ be a regular ideal of $B$ such that $f^{-1}(\f b)$ is a regular ideal of $A$.  Then $A\bowtie^f \f b $ is a Pr\"ufer ring if and only if $A,B$ are Pr\"ufer rings and $\f b=B$.
\end{corollary}
\begin{proof}
 Apply Proposition \ref{prufer-regular}, keeping in mind the representation of $A\bowtie^f \f b $ as a bi-amalgamation (see \cite[Example 2.1]{KLT}). 
\end{proof}
\begin{Prop}
Let $A/\f k$ be a total ring of fractions, where as usual $\f k:=\Ker(f)\cap \Ker(g)\subseteq \gi_0$, and assume that $\gb \subseteq \Jac(B)$ and $\gc \subseteq \Jac(C)$. If both $\gb$ and $\gc$ are torsion $A/\f k $-modules (with the $A/\f k$-module structure inherited by $f$ and $g$ respectively), then $\RR$ is a total ring of fractions.
\end{Prop}

\begin{proof}
Let $(f(a)+b,g(a)+c)\in \RR$ be a non-invertible element. We have to show that $(f(a)+b,g(a)+c)$ is a zero-divisor. Since $\gb \subseteq \Jac(B)$ and $\gc \subseteq \Jac(C)$, the maximal ideals of $\RR$ are exactly those of the form $\gp\bowtie^{f,g}(\gb,\gc)$, where $\gp \in \Max(A)\cap V(\gi_0)$. Hence, there exists a maximal ideal $\gm$ of $A$ containing $\gi_0$ such that $(f(a)+b,g(a)+c)\in \gm\bowtie^{f,g}(\gb,\gc)$. In particular, $a \in \gm$ is a zero-ivisor modulo $\f k$, so there exists $a'\notin \f k$ such that $aa'\in \f k$. Since  both $\gb$ and $\gc$ are torsion $A/\f k$-modules, there exist two regular elements $x_0+\f k$ and $y_0+\f k$ in $A/\f k$ such that $f(x_0)b=0$ and $g(y_0)c=0$. Of course, $a'x_0y_0 \notin \f k$, and so $(f(x_0y_0),g(x_0y_0))$ is a non-zero element such that $(f(a)+b,g(a)+c)(f(x_0y_0),g(x_0y_0))=(0,0)$.
\end{proof}

\begin{Lemma}
Assume that $(f(a)+b,g(a)+c)$ is a zero-divisor of $\RR$. Then, at least one of the following conditions hold:
\begin{enumerate}
\item
$a+\gi_0$ is a zero-divisor of $A/\gi_0$;
\item
there exists $(b',c')\in \gb\times \gc$, with $(b',c')\neq (0,0)$, such that $b'(f(a)+b)=0$ and $c'(g(a)+c)=0$.
\end{enumerate}

\end{Lemma}

\begin{proof}
Assume that $(f(a)+b,g(a)+c)(f(a')+x,g(a')+y)=0$ for some non-zero element $(f(a')+x,g(a')+y)$ of $\RR$. Then $(f(a)+b)(f(a')+x)=0$ implies that $aa' \in \gi_0$. If $a' \notin \gi_0$, then $a+\gi_0$ is a zero-divisor of $A/\gi_0$. Otherwise, $f(a')+x \in \gb$, $g(a')+y\in \gc$ and at least one of them is not zero.
\end{proof}

Notice that condition $(2)$ of the previous Lemma always implies that $(f(a)+b,g(a)+c)$ is a zero-divisor of $\RR$. If also condition $(1)$ implies that $(f(a)+b,g(a)+c)$ is a zero-divisor of $\RR$, we say that $\RR$ has the condition $(\star)$.

\begin{example}
Assume that $B$ and $C$ are local total quotient rings with maximal ideals $\gb$ and $\gc$, respectively. Let $f:A\longrightarrow B$ and $g: A\longrightarrow C$ be ring homomorphisms such that $\gi_0=f^{-1}(\gb)=g^{-1}(\gc)$. Then condition $(\star)$ holds.
\end{example}

\begin{Thm}

$(1)$ Assume that condition (\ref{*}) holds. If $\RR$ is a Pr\"ufer ring, then $A/\gi_0$ is a Pr\"ufer ring and $\gb B_{S_\gm}=f_\gm(\frac{r}{1})\gb B_{S_\gm}$ and $\gc C_{T_\gm}=g_\gm(\frac{r}{1})\gc C_{T_\gm}$ for every $\gm \in \Max(A)\cap V(\gi_0)$ and every element $r \in A$ such that $r+\gi_0$ is a regular element in $A/\gi_0$.

$(2)$ Assume that $\RR$ is a local ring having condition $(\star)$. Let $\gm$ be the unique maximal ideal of $A$ containing $\gi_0$. If $A/\gi_0$ is a Pr\"ufer ring and $\gb=f(r)\gb$ and $\gc=g(r)\gc$ for every $r\in \pi^{-1}(\Reg(A/\gi_0))$, then $\RR$ is a Pr\"ufer ring.
\end{Thm}

\begin{proof}
Part $(1)$ follows from Proposition \ref{gauss} and Remark \ref{remark1}.

$(2)$ Let $(f(a_1)+b_1,g(a_1)+c_1), (f(a_2)+b_2,g(a_2)+c_2) \in \RR$ and assume that $(f(a_1)+b_1,g(a_1)+c_1)$ is a regular element.
Condition $(\star)$ implies that $a_1+\gi_0$ is a regular element of the Pr\"ufer local ring $A/\gi_0$, and so the principal ideals of $A/\gi_0$ generated by $a_1+\gi_0$ and $a_2+\gi_0$ are comparable. There are two cases.

Case 1. There exist $x \in A$ and $u \in \gi_0$ such that $a_2=a_1x+u$. We can write $b_1=f(a_1)b_1'$ for some $b_1'\in \gb$. Notice that in our hypothesis $\gb\subseteq \Jac(B)$, so in particular, $1+b_1'$ is an invertible element of $B$. It follows that we can also find $\beta \in \gb$ such that
$$
f(a_1)\beta=\frac{b_2+f(u)-b_1f(x)}{1+b_1'}.
$$
Elements $c_1'$ and $\gamma$ in $\gc$ can be defined in an analogous way, getting
$$
g(a_1)\gamma=\frac{c_2+g(u)-c_1g(x)}{1+c_1'}.
$$
It is now straightforward to show that $(f(a_2)+b_2,g(a_2)+c_2)=(f(a_1)+b_1,g(a_1)+c_1)(f(x)+\beta,g(x)+\gamma)$, that is, the principal ideals of $\RR$ generated by $(f(a_1)+b_1,g(a_1)+c_1)$ and $(f(a_2)+b_2,g(a_2)+c_2)$ are comparable.

Case 2. There exist $x \in A$ and $u \in \gi_0$ such that $a_1=a_2x+u$. Notice that $\gb=f(a_2x)\gb$ and $\gc=g(a_2x)\gc$, because also $a_2x+\gi_0$ is a regular element of $A/\gi_0$. As before, we can find elements $b_2'\in \gb$ and $c_2' \in \gc$ such that $b_2=f(a_2x)b_2'$, $c_2=g(a_2x)c_2'$. Moreover, there exist $\beta \in \gb$ and $\gamma \in \gc$ such that
$$
f(a_2x)\beta=\frac{b_1+f(u)-b_2f(x)}{1+f(x)b_2'}
$$
and
$$
g(a_2x)\gamma=\frac{c_1+g(u)-c_2g(x)}{1+g(x)c_2'}.
$$
Now, we can conclude noting that the following equality holds:
$$
(f(a_1)+b_1,g(a_1)+c_1)=(f(a_2)+b_2,g(a_2)+c_2)(f(x)+\beta f(x),g(x)+\gamma g(x)).
$$
\end{proof}

\section{Gaussian bi-amalgamations}
We are going to study the transfer of Gauss condition to bi-amalgamations. 
We largely use the following characterization of local Gauss rings (\cite[Theorem 2.2]{HT}). A local ring $S$ is Gaussian if and only if for every two elements $x,y\in S$ the following two conditions hold:

(i) $(x,y)^2=(x^2)$ or $(y^2)$;

(ii) if $(x,y)^2=(x^2)$ and $xy=0$, then $y^2=0$.

\begin{Lemma}\label{idquad}
Let $\ga$ be an ideal of a Gaussian local ring $S$. Then $\ga^2=0$ if and only if $a^2=0$ for every $a \in \ga$. 
\end{Lemma}
\begin{proof}
For every $x, y \in \ga$, the ideal $(x,y)^2$ is equal to $(x^2)$ or $(y^2)$ \cite[Theorem 2.2]{HT}. If $a^2=0$ for every $a \in \ga$, we have $(xy)=(x,y)^2=0$, which implies that $xy=0$ and so $\ga^2=0$.
\end{proof}
\begin{Thm}\label{Gaussnec}
Assume that $\RR$ is a Gaussian local ring. Then:
\begin{enumerate}
\item
$A/\gi_0, f(A)+\gb$ and $g(A)+\gc$ are Gaussian local rings;
\item
if $\gb^2\neq 0$, then $\gc^2=0$;
\item
if $\gb^2=0$ and $f$ is surjective, then $f(a)\gb\subseteq f(a^2)B$ for every $a \in A$.
\end{enumerate}
\end{Thm}

\begin{proof}
$(1)$ immediately follows from the fact that quotients of a Gaussian ring are Gaussian rings.

$(2)$ Assume that $\gb^2\neq 0$. Since $\gb$ is an ideal of $f(A)+\gb$, by Remark \ref{idquad} there exists an element $b \in \gb$ such that $b^2\neq 0$. For any element $c \in \gc$ we have that the ideal $((b,0),(0,c))^2=((b^2,0),(0,c^2))$ of $\RR$ must be equal to $((b^2,0))$ or $((0,c^2))$. Since $b^2\neq 0$, $((b^2,0),(0,c^2))=((b^2,0))$ must be the case, which implies that $c^2=0$. We can conclude applying Remark \ref{idquad}.

$(3)$ Assume that $\gb^2=0$. For any element $a \in A$ we have that the ideal $((f(a),g(a)),(b,0))^2$ is equal to $0$ or $(f(a^2),g(a^2))$. In the first case $f(a)b=0$, while in the second one, there exist $\alpha \in A,\ \beta \in \gb$ and $\gamma \in \gc$ such that $f(a)b=f(a^2)(f(\alpha)+\beta)$ and $0=g(a^2)(g(\alpha)+\gamma)$. In both cases $f(a)b \in f(a^2)B$.
\end{proof}
The following example shows that the converse of Theorem \ref{Gaussnec} does not hold. 
\begin{example}
	Let $p$ be a prime number, let $A:=\mathbb Z_{(p)}, B:=\mathbb Z_{(p)}/p^4\mathbb Z_{(p)}$, let $f:A\longrightarrow B$ be the canonical projection and let $\f b$ be the principal ideal of $B$ generated by the class of $p^2$. By \cite[Example 2.1]{KLT}, the bi-amalgamation $A\bowtie^{f,{\rm Id}_A}(\f b, \f i_0)$ is the standard amalgamation $A\bowtie^f \f b$ and, since $f(p)\f b\neq f(p^2)\f b$, it is not Gaussian, by \cite[Theorem 4.1]{ASS}, but all the conditions of Theorem \ref{Gaussnec} are trivially satisfied. 
\end{example}
\begin{Thm}\label{Gausssuf}
Assume that $f$ and $g$ are surjective. Then $\RR$ is a Gaussian local ring whenever the following conditions hold:
\begin{enumerate}
\item
$A$ is a Gaussian local ring;
\item
$\gb^2=\gc^2=0$;
\item
$f(a)\gb=f(a^2)\gb$ and $g(a)\gc=g(a^2)\gc$ for every $a \in A$.
\end{enumerate}
\end{Thm}

\begin{proof}
First notice that conditions $(1)$ and $(2)$ imply that $\RR$ is a local ring (Proposition \ref{prime-spectrum}$(4)$). Moreover, since $A$ is Gaussian and $f$ and $g$ are surjective, then $B$ and $C$ are Gaussian.
Consider elements $(f(a_1)+b_1,g(a_1)+c_1)$ and $(f(a_2)+b_2,g(a_2)+c_2)$ of $\RR$. Since $A$ is Gaussian, we can assume that the ideal $(a_1,a_2)^2$ of $A$ is equal to $(a_1)^2$, and so there exist elements $x,y \in A$ such that $a_2^2=a_1^2x$ and $a_1a_2=a_1^2y$. We want to show that
the ideal $((f(a_1)+b_1,g(a_1)+c_1),(f(a_2)+b_2,g(a_2)+c_2))^2$ of $\RR$ is equal to $((f(a_1)+b_1,g(a_1)+c_1)^2)$. In order to do that, we want to prove that there exist elements $\beta,\beta' \in \gb$ and $\gamma,\gamma' \in \gc$ such that 
$$
(f(a_2)+b_2,g(a_2)+c_2)^2=(f(a_1)+b_1,g(a_1)+c_1)^2(f(x)+\beta,g(x)+\gamma)
$$
and
$$
(f(a_1)+b_1,g(a_1)+c_1)(f(a_2)+b_2,g(a_2)+c_2)=$$
$$=(f(a_1)+b_1,g(a_1)+c_1)^2(f(y)+\beta',g(y)+\gamma').
$$
Using the fact that $\gb^2=\gc^2=0$ and the relation $a_2^2=a_1^2x$, the first equality can be rewritten as
$$
(2f(a_2)b_2, 2g(a_2))=(f(a_1^2)\beta+2f(a_1x)b_1, g(a_1^2)\gamma+2g(a_1x)c_1).
$$
By $(3)$, we have $2f(a_2)b_2=f(a_2^2)b_2'=f(a_1^2x)b_2'$  and $2g(a_2)c_2=g(a_2^2)c_2'=g(a_1^2x)c_2'$ for suitable elements $b_2'\in \gb$ and $c_2' \in \gc$. So, it suffices that $\beta$ and $\gamma$ satisfy the equality
$$
(f(a_1^2)\beta, g(a_1^2)\gamma)=(f(a_1)(f(a_1x)b_2'-2f(x)b_1), g(a_1)(g(a_1x)c_2'-2g(x)c_1)).
$$
The existence of the elements $\beta$ and $\gamma$ is now obvious, again by $(3)$.
For the second equality, we can argue in a similar way.
Now, we want to prove that if $((f(a_1)+b_1,g(a_1)+c_1),(f(a_2)+b_2,g(a_2)+c_2))^2=((f(a_1)+b_1,g(a_1)+c_1)^2)$ and $(f(a_1)+b_1,g(a_1)+c_1)(f(a_2)+b_2,g(a_2)+c_2)=0$, then $(f(a_2)+b_2,g(a_2)+c_2)^2=0$. Looking only at the first component, we have that the ideal $(f(a_1)+b_1,f(a_2)+b_2)^2$ of $B$ is equal to $(f(a_1)+b_1)^2$ and that the product $(f(a_1)+b_1)(f(a_2)+b_2)=0$. Since $B$ is Gaussian, it implies that $(f(a_2)+b_2)^2=0$. Similarly, $(g(a_2)+c_2)^2=0$.
To conclude, apply \cite[Theorem 2.2]{HT}. 
\end{proof}

\begin{remark}
Notice that Example \ref{example1} shows that condition $(1)$ in Theorem \ref{Gausssuf} is not necessary. To prove that conditions (2) and (3) are not necessary for $\RR$ to be a Gauss ring, let $p$ be a prime number, $A:=\mathbb Z_{(p)}$, $B:=\mathbb Z_{(p)}/p^2\mathbb Z_{(p)}$, let $f:A\longrightarrow B$ be the canonical projection and let $\f b$ (resp., $\f m$) be the maximal ideal of $B$ (resp., of $A$). By \cite[Example 2.1]{KLT}, the bi-amalgamation $A\bowtie^{f,{\rm Id}_A}(\f b,\f m)$ is the amalgamated algebra $A\bowtie^f\f b$ and it is Gauss, in view of \cite[Theorem 4.1]{ASS}. But $\f m^2\neq 0$ and, for instance, $p\f m\neq p^2\f m$. 
\end{remark}

\begin{example}
Let $(V,\gm)$ be a valuation domain such that $\gm^2\neq \gm^3$. Set $A:=V/\gm^3$, $B=C:=V/\gm^2$ and consider the canonical morphisms $f:A\longrightarrow B$ and $g:A\longrightarrow C$. If $\gb$ and $\gc$ are the maximal ideals of $B$ and $C$ respectively and $\gi_0=f^{-1}(\gb)=g^{-1}(\gc)$ is the maximal ideal of $A$, then it is immediate to check that the conditions of Theorem \ref{Gausssuf} are satisfied, and so $\RR$ is a Gaussian local ring.
\end{example}

\bibliographystyle{amsalpha}

\end{document}